\documentclass[11pt]{article}
\usepackage{amsmath,amssymb,amsthm}
\usepackage{url}
\usepackage[margin=1in]{geometry}
\theoremstyle{plain}
\newtheorem{theorem}{Theorem}[section]
\newtheorem{lemma}[theorem]{Lemma}
\newtheorem{corollary}[theorem]{Corollary}
\newtheorem{proposition}[theorem]{Proposition}
\theoremstyle{definition}

\theoremstyle{remark}
\newtheorem{remark}[theorem]{Remark}
\newcommand{\PP}{\mathcal{P}}
\numberwithin{equation}{section}
\begin{document}

\title{A Symbolic Differentiation--Summation Framework for Power Sums of
Arithmetic Sequences}
\author{Assist.\ Lect.\ Ahmed Abdlmohsin Abdlsahib\\[2pt]\small Polytechnic College of Samawah,\\ Al-Furat Al-Awsat Technical University, Samawah, Iraq\\ ins.ahmd@atu.edu.iq}
\date{}
\maketitle

\begin{abstract}
We give a self-contained derivation of the closed form of the power sums
$S_m(n)=\sum_{i=1}^{n}(a_0+(i-1)d)^m$ in which the Bernoulli numbers enter as
constants of integration rather than as prescribed coefficients. The
derivation rests on an elementary fact: a sum whose terms are polynomial in
the index is the restriction to the integers of a unique polynomial, which
satisfies the defining difference equation identically and may therefore be
differentiated and integrated exactly. Differentiating lowers the power by
one; integrating back introduces a single constant, equal to the slope
$S_m'(0)$ of the polynomial at the origin and fixed by the value
$S_m(1)=a_0^m$. This yields a closed form for $S_m(n)$ in terms of the constants and a
recurrence that determines them, obtained without induction on the exponent
and without generating functions. With the Bernoulli numbers defined by their classical recurrence, we
identify the constants as $d^mB_m(a_0/d)$; the progression starting at $0$
reproduces that recurrence exactly. In
the normalized case they are $B_m(1)$, so $C_1=B_1(1)=+\tfrac12$, a value
that reflects the summation range; a reflection identity shows, without
generating functions, that the odd constants beyond the first vanish. The
coefficients of the recurrence are scaled binomial coefficients from a single
row of Pascal's triangle, and the same approach gives closed forms for
sequences with polynomial differences. The integral recursion behind the
method is known in the normalized case; each result is related to its
classical counterpart, and a verified implementation is provided.
\end{abstract}


\section{Introduction}
\label{sec:intro}

Formulas for the power sums $\sum_{i=1}^{n}i^m$ have a long history.
Faulhaber, in his \emph{Academia Algebrae} (Augsburg, 1631), published
formulas for the sums of odd powers up to the seventeenth, expressed as
polynomials in $N=n(n+1)/2$, and believed that such polynomials exist for
every odd power; a rigorous proof of this assertion was first published by
Jacobi in 1834 \cite{Knuth}. Jacob Bernoulli's
\emph{Ars Conjectandi}, published posthumously in 1713, contains a table of
power sums together with the rule that generates its coefficients
\cite[pp.~97--98]{Bernoulli}. In modern form the coefficients are expressed
through the Bernoulli numbers and polynomials; for an arithmetic progression
with $d\ne0$,
\[
  \sum_{k=0}^{n-1}(a+dk)^m=\frac{d^m}{m+1}\Bigl(B_{m+1}\bigl(n+\tfrac ad\bigr)
  -B_{m+1}\bigl(\tfrac ad\bigr)\Bigr)
  \qquad\text{\cite[Eq.~(24.4.9)]{DLMF}},
\]
and the Euler--Maclaurin formula relates finite sums to integrals through
correction terms whose coefficients are Bernoulli numbers
\cite[Eq.~(2.10.1)]{DLMF}.

In most modern treatments the Bernoulli numbers enter as given
coefficients, defined beforehand by a generating function or by a
recurrence. A different viewpoint obtains them from the power sums
themselves: the integral recursion
$S_k(x)=k\int_0^xS_{k-1}(t)\,dt+\bigl(1-k\int_0^1S_{k-1}(t)\,dt\bigr)x$,
$S_0(x)=x$, for the normalized sums is classical
\cite{Cereceda,Coen}, and its linear coefficient is the Bernoulli number
$B_k(1)$. The purpose of this paper is to give this viewpoint a rigorous
foundation, to extend it to an arbitrary arithmetic progression, and to show
how the Bernoulli numbers arise as constants of integration when a power sum
is reconstructed from its derivative. The starting point is
elementary. A sum whose terms are polynomial in the index is the restriction
to the integers of a unique polynomial, and this polynomial satisfies the
defining difference equation $S(n+1)-S(n)=x(n+1)$ identically. That identity
can be differentiated: the derivative of a power sum is a power sum of one
degree less, plus a constant. Integrating back introduces exactly one
constant, which equals the slope $S'(0)$ of the polynomial at the origin and
is determined by the single value $S(1)$.

The contributions of the paper are the following.
\begin{enumerate}
\item[(i)] A justification for differentiating and integrating finite sums
with polynomial terms, based on the uniqueness of the polynomial extension
(Section~\ref{sec:why}). No sum with a variable upper limit is
differentiated, and no approximation is involved.
\item[(ii)] The Recursive Bernoulli Extraction Formula (RBEF)
(Section~\ref{sec:rbef}): a closed form for
$S_m(n)=\sum_{i=1}^{n}(a_0+(i-1)d)^m$ in terms of the constants
$C_m=S_m'(0)$ and a recurrence that determines them, derived without
induction on the exponent and without generating functions; and the
identification $C_m=d^mB_m(a_0/d)$, with the Bernoulli numbers defined by
their classical recurrence.
\item[(iii)] An explanation of the value $C_1=+\tfrac12$ in terms of the
summation range (Section~\ref{sec:rbef}), and a proof that the odd constants
vanish in the normalized case based on a reflection identity, without
generating functions (Section~\ref{sec:odd}).
\item[(iv)] A description of the coefficients of the recurrence as scaled
binomial coefficients (Section~\ref{sec:row}), a procedure for sequences
with polynomial differences (Section~\ref{sec:polydiff}), and a verified
implementation (Section~\ref{sec:verif} and Appendix~\ref{app:code}).
\end{enumerate}

The formulas obtained are classical, and the precise correspondences with the
literature are stated where they arise; in the normalized case the integral
recursion itself is known \cite{Cereceda,Coen}. What the paper adds is a
justification of the differentiation and integration of finite sums, the
extension to an arbitrary arithmetic progression, and derivations in which
each constant has a definite meaning: it is a Taylor coefficient of the
polynomial extension at the origin.

\section{Related Work}
\label{sec:related}

Faulhaber's formulas and their later history, including Jacobi's proof, are
discussed by Knuth \cite{Knuth}; Bernoulli's original treatment is in
\cite[pp.~96--98]{Bernoulli}. The standard identities for Bernoulli numbers and
polynomials used in this paper are collected in \cite[Ch.~24]{DLMF}.

A different line of work treats summation algorithmically. Gosper
introduced a decision procedure for indefinite hypergeometric summation that
finds a hypergeometric antidifference when one exists \cite{Gosper}.
Zeilberger turned creative telescoping (a term he attributes to van der
Poorten) into an algorithm that produces recurrences and proof certificates
for definite hypergeometric sums \cite{Zeilberger}.
Petkov\v{s}ek, Wilf and Zeilberger gave a systematic treatment of these
methods \cite[Chaps.~4--8]{AeqB}, and Schneider developed difference-field
algorithms, building on Karr's algorithm and implemented in the package
Sigma, for nested indefinite sums and
products, for deriving recurrences, and for verifying identities
\cite[Secs.~2, 4--6]{Schneider}. These algorithms operate on a given
summand; applied to $\sum_{i=1}^{n}i^m$ for a fixed $m$, Gosper's algorithm
returns the corresponding polynomial (we checked this with a computer-algebra
implementation for $m\le12$). The present approach has a different
aim: it produces one recursive structure valid for all $m$, in which the
Bernoulli numbers appear as integration constants.

Closest to the present paper is the integral recursion for the normalized
power sums. Cereceda \cite{Cereceda} describes it as well established and
credits it to Sherwood (1970) and Costa (2008), and Coen \cite[\S3.4]{Coen}
derives it following Johnson (1986). The present paper differs from these
works in the justification of Section~\ref{sec:why}, which applies to an
arbitrary polynomial summand, and in treating a general arithmetic
progression.

\section{Polynomial Extension of Finite Sums}
\label{sec:why}

A finite sum $S(n)=\sum_{i=1}^{n}x(i)$ is defined only for integers
$n\ge 0$; an expression such as $\sum_{i=1}^{t}x(i)$ with real $t$ is not
defined, and neither, a priori, are $dS/dn$ and $\int S\,dn$. When the terms
are polynomial in the index, the sum is the restriction to the integers of a
unique polynomial, which may be differentiated and integrated with exact
results at every integer. This section establishes these facts; they are the
basis of all subsequent constructions, and no sum with a variable upper
limit is differentiated anywhere in the paper.

Throughout, $\PP_k$ denotes the real polynomials of degree at most $k$, and
$(\Delta p)(t)=p(t+1)-p(t)$ is the forward difference.

\begin{lemma}\label{lem:kernel}
If $p$ is a polynomial and $\Delta p\equiv 0$, then $p$ is constant.
\end{lemma}
\begin{proof}
From $p(t+1)=p(t)$ for all $t$ we get $p(j)=p(0)$ for every integer $j$, so
the polynomial $p(t)-p(0)$ has infinitely many zeros and therefore vanishes
identically.
\end{proof}

\begin{lemma}\label{lem:onto}
For every $k\ge 0$, $\Delta$ maps $\PP_{k+1}$ onto $\PP_k$.
\end{lemma}
\begin{proof}
For $j\ge1$, $\Delta t^{j}=(t+1)^j-t^j=\sum_{i=0}^{j-1}\binom{j}{i}t^{i}$ has
degree $j-1$, and $\Delta 1=0$; hence $\Delta(\PP_{k+1})\subseteq\PP_k$. The
map $\Delta:\PP_{k+1}\to\PP_k$ is linear, and by Lemma~\ref{lem:kernel} its
kernel consists of the constants, a space of dimension $1$. By the
rank--nullity theorem its image has dimension
$\dim\PP_{k+1}-1=(k+2)-1=k+1=\dim\PP_k$, so the map is onto.
\end{proof}

\begin{theorem}[Polynomial extension of a sum]\label{thm:ext}
Let $x\in\PP_k$ and $S(n)=\sum_{i=1}^{n}x(i)$ for integers $n\ge0$. There is
exactly one polynomial $\widehat S$ such that
\begin{equation}\label{eq:defS}
  \widehat S(t+1)-\widehat S(t)=x(t+1)\quad\text{for all real }t,
  \qquad \widehat S(0)=0 .
\end{equation}
It has degree at most $k+1$, it satisfies $\widehat S(n)=S(n)$ for every
integer $n\ge0$, and it is the only polynomial that agrees with $S$ at all
integers $n\ge0$.
\end{theorem}
\begin{proof}
\emph{Existence.} The polynomial $y(t)=x(t+1)$ lies in $\PP_k$, so by
Lemma~\ref{lem:onto} there is $Q\in\PP_{k+1}$ with $\Delta Q=y$. Then
$\widehat S=Q-Q(0)$ satisfies \eqref{eq:defS}.

\emph{Uniqueness.} If two polynomials satisfy \eqref{eq:defS}, their
difference has zero forward difference, hence is constant by
Lemma~\ref{lem:kernel}, and it vanishes at $0$.

\emph{Agreement with the sum.} Setting $t=i-1$ in \eqref{eq:defS} gives
$\widehat S(i)-\widehat S(i-1)=x(i)$. Summing over $i=1,\dots,n$, the left
side telescopes:
$\widehat S(n)=\widehat S(0)+\sum_{i=1}^{n}x(i)=S(n)$.

\emph{Only extension.} A polynomial that agrees with $S$ at all integers
$n\ge 0$ differs from $\widehat S$ by a polynomial with infinitely many
zeros.
\end{proof}

\begin{remark}\label{rem:why}
(i) Uniqueness holds only among polynomials. The function
$\widehat S(t)+\sin(\pi t)$ is analytic and agrees with $S$ at every
integer, but its derivative differs from $\widehat S'$. This is why the
extension must be the polynomial one, and why differentiating an identity
that is known only at the integers is not legitimate.

(ii) $\widehat S$ is defined by the identity \eqref{eq:defS}, which holds for
all real $t$; it is not defined as a sum with a variable upper limit. The
``moving boundary'' therefore never arises.

(iii) The condition $\widehat S(0)=0$ is not an arbitrary normalization:
the empty sum $S(0)$ equals $0$, and $\widehat S$ must agree with $S$ at
$n=0$.
\end{remark}

In what follows we write $S$ for $\widehat S$, and, for a polynomial $y$,
$\Sigma[y]$ for the polynomial extension of $n\mapsto\sum_{i=1}^{n}y(i)$. By
uniqueness, $\Sigma$ is linear:
$\Sigma[\alpha y+\beta z]=\alpha\Sigma[y]+\beta\Sigma[z]$.

\begin{theorem}[Differentiation]\label{thm:diff}
Let $x$ be a polynomial and $S=\Sigma[x]$. For every $r\ge1$,
\begin{equation}\label{eq:DI}
  S^{(r)}(t)=\Sigma\bigl[x^{(r)}\bigr](t)+S^{(r)}(0)\qquad\text{for all real }t .
\end{equation}
\end{theorem}
\begin{proof}
Since \eqref{eq:defS} is an identity between polynomials, it may be
differentiated $r$ times:
$S^{(r)}(t+1)-S^{(r)}(t)=x^{(r)}(t+1)$. The polynomial
$U=\Sigma[x^{(r)}]$ satisfies the same relation. Hence $\Delta(S^{(r)}-U)=0$,
so $S^{(r)}-U$ is constant by Lemma~\ref{lem:kernel}; evaluating at $t=0$,
where $U(0)=0$, gives the constant $S^{(r)}(0)$.
\end{proof}

\begin{theorem}[Integration: reconstruction of the sum]\label{thm:int}
Let $x$ be a polynomial and $U=\Sigma[x']$. Then for every integer $n\ge0$
\begin{equation}\label{eq:recon}
  \sum_{i=1}^{n}x(i)=\int_{0}^{n}U(t)\,dt+c\,n,
  \qquad
  c=x(1)-\int_{0}^{1}U(t)\,dt .
\end{equation}
Moreover $c=S'(0)$.
\end{theorem}
\begin{proof}
By the fundamental theorem of calculus applied to the polynomial $S$,
$S(n)=S(0)+\int_0^n S'(t)\,dt$. By Theorem~\ref{thm:diff} with $r=1$,
$S'=U+S'(0)$, and $S(0)=0$. Hence $S(n)=\int_0^n U(t)\,dt+S'(0)\,n$ for all
real $n$. Taking $n=1$ and using $S(1)=x(1)$ gives the stated value of
$c=S'(0)$. Finally $S(n)=\sum_{i=1}^{n}x(i)$ at every integer $n\ge0$ by
Theorem~\ref{thm:ext}.
\end{proof}

\begin{remark}
Theorem~\ref{thm:int} justifies the use of integration throughout the
paper. The integral does not approximate the sum (indeed
$\int_0^n t\,dt\ne\sum_{i=1}^{n}i$). What is integrated is the derivative of
the polynomial that coincides with the sum at the integers, and on
polynomials integration is exact. For polynomial summands the Euler--Maclaurin formula is
also exact, since its remainder vanishes; the difference is that there the
Bernoulli numbers enter as given coefficients of the correction terms,
whereas here the single constant $c$ is determined by the value $S(1)$. The constant of integration is
$c=S'(0)$, the slope of the polynomial extension at the origin, and it is
determined by the single value $S(1)=x(1)$.
\end{remark}

\begin{corollary}[Power sums of an arithmetic progression]\label{cor:power}
Let $S_m(n)=\sum_{i=1}^{n}\bigl(a_0+d(i-1)\bigr)^m$, with $S_0(n)=n$. For
every $m\ge1$,
\begin{equation}\label{eq:RBEFint}
  S_m(n)=m\,d\int_0^{n}S_{m-1}(t)\,dt+C_m\,n,
  \qquad
  C_m=a_0^{m}-m\,d\int_0^{1}S_{m-1}(t)\,dt=S_m'(0).
\end{equation}
\end{corollary}
\begin{proof}
Apply Theorem~\ref{thm:int} to $x(t)=(a_0+d(t-1))^m$. Then
$x'(t)=m\,d\,(a_0+d(t-1))^{m-1}$, so $U=\Sigma[x']=m\,d\,S_{m-1}$, and
$x(1)=a_0^m$.
\end{proof}

\begin{remark}
Equation \eqref{eq:RBEFint} is the RBEF in integral form: starting from
$S_0(n)=n$ it produces $S_1,S_2,\dots$ and the constants $C_1,C_2,\dots$
one after another, each step requiring one polynomial integration and one
evaluation at $n=1$. Inserting the closed form of $S_{m-1}$ into the
integral $\int_0^1 S_{m-1}(t)\,dt$ yields the recurrence for $C_m$ in
Corollary~\ref{cor:rbef}. For example, with $a_0=d=1$ and $m=2$: $S_1(t)=t(t+1)/2$,
$\int_0^1 S_1=5/12$, hence $C_2=1-2\cdot\tfrac{5}{12}=\tfrac16$ and
$S_2(n)=\tfrac{n^3}{3}+\tfrac{n^2}{2}+\tfrac{n}{6}$.
\end{remark}

\section{The Recursive Bernoulli Extraction Formula}
\label{sec:rbef}

Throughout this section $S_m(n)=\sum_{i=1}^{n}\bigl(a_0+d(i-1)\bigr)^m$ for
$m\ge0$, with the convention $0^0=1$ (so $S_0(n)=n$), always identified with its polynomial extension
(Theorem~\ref{thm:ext}), and
\begin{equation}\label{eq:Cdef}
  C_m:=S_m'(0)\qquad(m\ge 0),\qquad C_0=1 .
\end{equation}
By Theorem~\ref{thm:int}, $C_m$ is the constant of integration that appears
when $S_m$ is reconstructed from its derivative.

\subsection{The first cases}

Corollary~\ref{cor:power} gives $S_m(n)=m\,d\int_0^nS_{m-1}(t)\,dt+C_m n$,
with $C_m$ fixed by $S_m(1)=a_0^m$. Starting from $S_0(n)=n$:
\begin{align}
S_1(n)&=\frac{d\,n^2}{2}+C_1n, & C_1&=a_0-\frac d2, \label{eq:S1}\\
S_2(n)&=\frac{d^2n^3}{3}+d\,C_1n^2+C_2n, & C_2&=a_0^2-\frac{d^2}{3}-d\,C_1, \label{eq:S2}\\
S_3(n)&=\frac{d^3n^4}{4}+d^2C_1n^3+\frac32 d\,C_2n^2+C_3n, &
C_3&=a_0^3-\frac{d^3}{4}-d^2C_1-\frac32 d\,C_2. \label{eq:S3}
\end{align}
For $a_0=d=1$ these give $C_1=\tfrac12$, $C_2=\tfrac16$, $C_3=0$ and the
classical formulas $\tfrac{n(n+1)}{2}$, $\tfrac{n(n+1)(2n+1)}{6}$,
$\bigl(\tfrac{n(n+1)}{2}\bigr)^2$.

\subsection{The general closed form}

\begin{theorem}[Closed form]\label{thm:closed}
For every $m\ge1$,
\begin{equation}\label{eq:closed}
  S_m(n)=\frac{1}{m+1}\sum_{j=0}^{m}\binom{m+1}{j}\,d^{\,m-j}\,C_j\,n^{\,m+1-j}.
\end{equation}
Equivalently, with $A(m,k)=\frac{1}{k+1}\binom{m}{k}$,
\begin{equation}\label{eq:closedA}
  S_m(n)=\frac{d^m n^{m+1}}{m+1}+d^{m-1}C_1n^m
  +\sum_{k=1}^{m-1}A(m,k)\,d^{\,m-k-1}\,C_{k+1}\,n^{m-k}.
\end{equation}
\end{theorem}
\begin{proof}
$S_m$ is a polynomial of degree at most $m+1$ with $S_m(0)=0$, so it equals
its Taylor expansion at $0$:
$S_m(n)=\sum_{r=1}^{m+1}S_m^{(r)}(0)\,n^r/r!$.
We compute each coefficient directly. For $x(t)=(a_0+d(t-1))^m$ and
$2\le r\le m+1$ we have
$x^{(r-1)}(t)=d^{\,r-1}\frac{m!}{(m-r+1)!}\,(a_0+d(t-1))^{m-r+1}$, so
Theorem~\ref{thm:diff} (of order $r-1$) gives
\[
  S_m^{(r-1)}(t)=d^{\,r-1}\frac{m!}{(m-r+1)!}\,S_{m-r+1}(t)+\text{const}.
\]
Differentiating once and setting $t=0$,
\[
  S_m^{(r)}(0)=d^{\,r-1}\frac{m!}{(m-r+1)!}\,C_{m-r+1},
\]
and this also holds for $r=1$ by \eqref{eq:Cdef}. Hence the coefficient of
$n^r$ is $d^{\,r-1}\frac{m!}{r!\,(m+1-r)!}C_{m+1-r}
=\frac{1}{m+1}\binom{m+1}{r}d^{\,r-1}C_{m+1-r}$. Writing $j=m+1-r$ gives
\eqref{eq:closed}. Since $\frac{1}{m+1}\binom{m+1}{k+1}=A(m,k)$, the terms
$j=0$, $j=1$ and $j=k+1$ ($1\le k\le m-1$) of \eqref{eq:closed} are the
three parts of \eqref{eq:closedA}.
\end{proof}

\begin{remark}
The proof uses Theorem~\ref{thm:diff} at every order $r\le m$; this is where
the identity ``for all orders'' is needed. No induction on $m$ is used: each
Taylor coefficient is computed directly from Theorem~\ref{thm:diff}.
\end{remark}

\subsection{The recurrence}

\begin{corollary}[RBEF]\label{cor:rbef}
For every $m\ge1$,
\begin{equation}\label{eq:rec-compact}
  \sum_{j=0}^{m}\binom{m+1}{j}\,d^{\,m-j}\,C_j=(m+1)\,a_0^m ,
\end{equation}
and for every $m\ge2$, with $C_1=a_0-\tfrac d2$,
\begin{equation}\label{eq:rec}
  C_m=a_0^m-\frac{d^m}{m+1}-d^{m-1}C_1
      -\sum_{k=1}^{m-2}A(m,k)\,d^{\,m-k-1}\,C_{k+1}.
\end{equation}
\end{corollary}
\begin{proof}
Set $n=1$ in \eqref{eq:closed} and use $S_m(1)=a_0^m$; this is
\eqref{eq:rec-compact}. For $m\ge2$ the terms $j=1$ and $j=m$ are distinct;
the term $j=m$ has coefficient $\binom{m+1}{m}=m+1$, and solving for $C_m$
and rewriting the remaining terms as in \eqref{eq:closedA} gives
\eqref{eq:rec}. (For $m=1$ the two terms coincide, and
\eqref{eq:rec-compact} gives $C_1=a_0-\tfrac d2$ directly.)
\end{proof}

Since the coefficient of $C_m$ in \eqref{eq:rec-compact} is $m+1\ne0$,
the recurrence determines $C_1,C_2,\dots$ uniquely, one after another; that
its solution is $S_m'(0)$ is Theorem~\ref{thm:closed}, which holds for each
$m$ separately. The constants are thus computed from $C_0=1$ by
\eqref{eq:rec-compact}, from $C_1=a_0-\tfrac d2$ by \eqref{eq:rec}, or
equivalently by the integral form
$C_m=a_0^m-m\,d\int_0^1S_{m-1}(t)\,dt$ of Corollary~\ref{cor:power}. No
generating function and no prior knowledge of Bernoulli numbers is used.
In the normalized case $a_0=d=1$, \eqref{eq:rec-compact} reads
$\sum_{j=0}^{m}\binom{m+1}{j}C_j=m+1$.

\subsection{Identification of the constants}

To identify the constants with Bernoulli numbers one must fix a definition
of the latter. We use the classical recurrence
\begin{equation}\label{eq:Bdef}
  B_0=1,\qquad \sum_{k=0}^{n-1}\binom nk B_k=0\quad(n\ge2),
\end{equation}
which is \cite[Eq.~(24.5.1)]{DLMF} at $x=0$ and gives
$B_1=-\tfrac12$, $B_2=\tfrac16$, $B_3=0,\dots$, together with the Bernoulli
polynomials $B_m(x)=\sum_{k=0}^{m}\binom mk B_kx^{m-k}$
\cite[Eq.~(24.2.5)]{DLMF}. These definitions are equivalent to the
generating-function definition \cite[\S24.2]{DLMF}, which is not needed for
any result of this paper. We write $C_m(a_0,d)$ for the constants of the
progression with first term $a_0$ and difference $d$.

\begin{theorem}\label{thm:ident}
For every $m\ge0$:
\begin{enumerate}
\item[(i)] $C_m(0,1)=B_m$;
\item[(ii)] for all $a_0$ and $d$,
$C_m(a_0,d)=\sum_{k=0}^{m}\binom mk a_0^{m-k}d^kB_k$, which equals
$d^m\,B_m\!\left(\frac{a_0}{d}\right)$ when $d\ne0$.
\end{enumerate}
\end{theorem}
\begin{proof}
(i) For $a_0=0$ and $d=1$, \eqref{eq:rec-compact} reads
$\sum_{j=0}^{m}\binom{m+1}{j}C_j(0,1)=(m+1)\,0^m=0$ for $m\ge1$, with
$C_0=1$. This is \eqref{eq:Bdef} with $n=m+1$. Since the coefficient of the
last unknown in \eqref{eq:Bdef} is $n\ne0$, its solution is unique, so
$C_m(0,1)=B_m$.

(ii) By the binomial theorem,
$(a_0+d(i-1))^m=\sum_{k=0}^{m}\binom mk a_0^{m-k}d^k(i-1)^k$, and
$\sum_{i=1}^{n}(i-1)^k$ is the power sum of the progression with $a_0=0$,
$d=1$. By linearity of $\Sigma$ (Section~\ref{sec:why}),
$S_m=\sum_k\binom mk a_0^{m-k}d^k\,S_k^{(0,1)}$, where $S_k^{(0,1)}$ denotes
that power sum. Differentiating at $0$ and using (i) gives
$C_m(a_0,d)=\sum_k\binom mk a_0^{m-k}d^kB_k$. For $d\ne0$ this is
$d^m\sum_k\binom mk B_k(a_0/d)^{m-k}=d^mB_m(a_0/d)$.
\end{proof}

\begin{remark}
No generating function enters the identification: part (i) is the observation that the progression starting
at $0$ reproduces the defining recurrence of the Bernoulli numbers
verbatim, and part (ii) is the binomial theorem.
\end{remark}

\begin{corollary}\label{cor:normal}
\begin{enumerate}
\item[(i)] For $a_0=d=1$: $C_m=B_m(1)=\sum_{k=0}^{m}\binom mk B_k$. Hence
$C_1=1+B_1=+\tfrac12$ and, since $\sum_{k=0}^{m-1}\binom mk B_k=0$ for
$m\ge2$ by \eqref{eq:Bdef}, $C_m=B_m$ for all $m\ge2$.
\item[(ii)] For $a_0=0$, $d=1$ (the sum $0^m+1^m+\dots+(n-1)^m$):
$C_m=B_m(0)=B_m$, in particular $C_1=-\tfrac12$.
\end{enumerate}
\end{corollary}

\begin{remark}[The value of $C_1$]
By \eqref{eq:S1}, $C_1=a_0-\tfrac d2$ is determined by the first term and the
common difference of the progression. The sum $1^m+\dots+n^m$ gives
$+\tfrac12$ and the sum $0^m+\dots+(n-1)^m$ gives $-\tfrac12$. The values
$B_1(1)=+\tfrac12$ and $B_1(0)=-\tfrac12$ thus correspond to the two
summation ranges; the framework explains the difference rather than choosing
between them.
\end{remark}

\begin{remark}[Relation to classical results]
Multiplying \cite[Eq.~(24.5.1)]{DLMF} (with $n=m+1$ and $x=a_0/d$) by $d^m$
gives exactly \eqref{eq:rec-compact} with $C_j=d^jB_j(a_0/d)$
(Theorem~\ref{thm:ident}), and \eqref{eq:closed} is the expansion of
\cite[Eq.~(24.4.9)]{DLMF} by the addition formula
\cite[Eq.~(24.4.12)]{DLMF}. The formulas of this section are
therefore classical; what the present derivation adds is that the
constants arise as integration constants $C_m=S_m'(0)$, computed from the
single condition $S_m(1)=a_0^m$.
\end{remark}

\begin{remark}[Numerical values]
For $a_0=d=1$ the recurrence \eqref{eq:rec} gives the values of
Table~\ref{tab:C}. For $2\le m\le16$ they agree with
\cite[Table~24.2.1]{DLMF}; the odd-index values vanish, in agreement with
\cite[Eq.~(24.2.2)]{DLMF} and Section~\ref{sec:odd}; for larger $m$ (up to
$m=300$, Section~\ref{sec:verif}) they agree with an independent exact
computation.
\end{remark}

\begin{table}[h]
\centering
\begin{tabular}{rr@{\qquad}rr}
\hline
$m$ & $C_m$ & $m$ & $C_m$\\ \hline
1 & $1/2$ & 12 & $-691/2730$\\
2 & $1/6$ & 14 & $7/6$\\
4 & $-1/30$ & 16 & $-3617/510$\\
6 & $1/42$ & 18 & $43867/798$\\
8 & $-1/30$ & 20 & $-174611/330$\\
10 & $5/66$ & odd $m\ge3$ & $0$\\ \hline
\end{tabular}
\caption{The constants $C_m$ for $a_0=d=1$, computed by \eqref{eq:rec}
from $C_1=\tfrac12$ (exact rational arithmetic).}
\label{tab:C}
\end{table}

\section{The Coefficient Row}
\label{sec:row}

The coefficients that enter the recurrence \eqref{eq:rec} are
\[
  A(m,k)=\frac{1}{k+1}\binom{m}{k},\qquad 1\le k\le m-2,\quad m\ge3,
\]
and we call $(A(m,1),\dots,A(m,m-2))$ the \emph{coefficient row} of order
$m$. (In \eqref{eq:closedA} the index runs to $k=m-1$; the extra value
$A(m,m-1)=1$ is the coefficient of $C_m$ itself.) All properties of the row
follow from one observation.

\begin{proposition}\label{prop:pascal}
For $0\le k\le m-1$,
\[
  A(m,k)=\frac{1}{m+1}\binom{m+1}{k+1}=\frac{1}{m-k}\binom{m}{k+1}.
\]
Thus the coefficient row of order $m$ consists of the entries
$\binom{m+1}{j}$, $2\le j\le m-1$, of row $m+1$ of Pascal's triangle,
divided by $m+1$ (Table~\ref{tab:rows}).
\end{proposition}
\begin{proof}
All three expressions equal $\dfrac{m!}{(k+1)!\,(m-k)!}$.
\end{proof}

\begin{corollary}\label{cor:rowprops}
Let $m\ge3$ and $1\le k\le m-2$.
\begin{enumerate}
\item[(i)] \emph{Symmetry:} $A(m,k)=A(m,m-1-k)$.
\item[(ii)] \emph{Ratio:} $\dfrac{A(m,k+1)}{A(m,k)}=\dfrac{m-k}{k+2}$ for
$k\le m-3$. Hence the row increases for $k<\frac{m-2}{2}$ and decreases for
$k>\frac{m-2}{2}$. Its maximum is attained only at $k=\frac{m-1}{2}$ if $m$
is odd, and at the two positions $k=\frac{m-2}{2}$ and $k=\frac m2$, with
equal values, if $m$ is even.
\item[(iii)] \emph{Catalan numbers:} $A(2k,k)=\frac{1}{k+1}\binom{2k}{k}
=\mathrm{Cat}_k$ \cite[Eq.~(26.5.1)]{DLMF}; within the range of the row this
occurs for $k\ge2$.
\item[(iv)] \emph{Integrality:} if $(k+1)\mid m$, then $A(m,k)\in\mathbb Z$.
The converse is false: $A(7,2)=7$, $A(6,4)=3$, $A(4,2)=2$.
\end{enumerate}
\end{corollary}
\begin{proof}
(i) By Proposition~\ref{prop:pascal} and
$\binom{N}{j}=\binom{N}{N-j}$ with $N=m+1$, $j=k+1$, since
$N-j=m-k=(m-1-k)+1$.
(ii) The ratio follows from the factorial form; it exceeds $1$ exactly when
$k<\frac{m-2}{2}$ and equals $1$ exactly when $k=\frac{m-2}{2}$, which is an
integer only for even $m$.
(iii) Set $m=2k$ in the definition.
(iv) From $k\binom{m}{k}=m\binom{m-1}{k-1}$, if $(k+1)\mid m$ then
$(k+1)\mid k\binom mk$; since $\gcd(k,k+1)=1$, $(k+1)\mid\binom mk$.
\end{proof}

\begin{table}[h]
\centering
\begin{tabular}{cl}
\hline
$m$ & $A(m,1),\dots,A(m,m-2)$\\ \hline
3 & $\tfrac32$\\
4 & $2,\ 2$\\
5 & $\tfrac52,\ \tfrac{10}{3},\ \tfrac52$\\
6 & $3,\ 5,\ 5,\ 3$\\
7 & $\tfrac72,\ 7,\ \tfrac{35}{4},\ 7,\ \tfrac72$\\
8 & $4,\ \tfrac{28}{3},\ 14,\ 14,\ \tfrac{28}{3},\ 4$\\
9 & $\tfrac92,\ 12,\ 21,\ \tfrac{126}{5},\ 21,\ 12,\ \tfrac92$\\
10 & $5,\ 15,\ 30,\ 42,\ 42,\ 30,\ 15,\ 5$\\ \hline
\end{tabular}
\caption{Coefficient rows for $3\le m\le10$ (exact values).}
\label{tab:rows}
\end{table}

\subsection{Computation}
By Corollary~\ref{cor:rowprops}(ii), a row is generated without factorials:
$A(m,1)=\frac m2$ and $A(m,k+1)=A(m,k)\,\frac{m-k}{k+2}$, one multiplication
and one division per entry; by symmetry only the first half need be
computed. The symmetry does not, however, halve the work in
\eqref{eq:rec}: the sum there needs every term, because the constants
$C_{k+1}$ on the two halves of the row differ. In the
normalized case $a_0=d=1$, the odd constants vanish (Section~\ref{sec:odd}),
so only the terms with $k$ odd contribute, about half of them. Computing
$C_1,\dots,C_M$ by \eqref{eq:rec} therefore takes $O(M^2)$ arithmetic
operations on rational numbers whose size grows with $M$.

\section{Vanishing of the Odd Constants}
\label{sec:odd}

In this section $a_0=d=1$, so $S_m(n)=\sum_{i=1}^{n}i^m$ and
$C_m=S_m'(0)$. The proof below uses only the uniqueness of the polynomial
extension (Theorem~\ref{thm:ext}) and the differentiation theorem
(Theorem~\ref{thm:diff}); no generating function is needed.

\begin{lemma}[Reflection]\label{lem:refl}
For $m\ge1$ and all real $n$,
\[
  S_m(-1-n)=(-1)^{m+1}S_m(n).
\]
Moreover $S_m(-1)=0$ for $m\ge1$, while $S_0(-1)=-1$.
\end{lemma}
\begin{proof}
Setting $t=-1$ in $S_m(t+1)-S_m(t)=(t+1)^m$ gives
$S_m(0)-S_m(-1)=0^m$, so $S_m(-1)=0$ for $m\ge1$; and $S_0(n)=n$.

Let $Q(n)=(-1)^{m+1}S_m(-1-n)$. Then $Q(0)=(-1)^{m+1}S_m(-1)=0$, and,
using the difference relation at $t=-n-2$,
\[
  Q(n+1)-Q(n)=(-1)^{m+1}\bigl(S_m(-n-2)-S_m(-n-1)\bigr)
  =(-1)^{m+1}\bigl(-(-n-1)^m\bigr)=(n+1)^m .
\]
Thus $Q$ satisfies the defining conditions \eqref{eq:defS} of $S_m$, and
$Q=S_m$ by Theorem~\ref{thm:ext}.
\end{proof}

Via Theorem~\ref{thm:ident}, Lemma~\ref{lem:refl} is equivalent to the
reflection formula \cite[Eq.~(24.4.3)]{DLMF}; the proof above uses only
Theorem~\ref{thm:ext}.

\begin{theorem}\label{thm:odd}
For $a_0=d=1$: $C_1=\tfrac12$, and $C_m=0$ for every odd $m\ge3$.
\end{theorem}
\begin{proof}
Differentiating Lemma~\ref{lem:refl} gives
$-S_m'(-1-n)=(-1)^{m+1}S_m'(n)$, and at $n=0$
\[
  C_m=S_m'(0)=(-1)^m\,S_m'(-1).
\]
By Theorem~\ref{thm:diff} with $r=1$, $S_m'(n)=m\,S_{m-1}(n)+C_m$, so
$S_m'(-1)=m\,S_{m-1}(-1)+C_m$. Hence
\[
  C_m=(-1)^m\bigl(m\,S_{m-1}(-1)+C_m\bigr).
\]
For $m=1$, $S_0(-1)=-1$ gives $C_1=1-C_1$, i.e.\ $C_1=\tfrac12$. For odd
$m\ge3$, $S_{m-1}(-1)=0$ gives $C_m=-C_m$, i.e.\ $C_m=0$.
\end{proof}

\begin{remark}[Generating-function proof]
The classical proof reaches the same conclusion through generating
functions. By Corollary~\ref{cor:rbef}, $\sum_{j=0}^{m}\binom{m+1}{j}C_j=m+1$
for all $m\ge0$ (trivially for $m=0$). Multiplying by $x^{m+1}/(m+1)!$ and summing gives, as formal
power series, $G(x)\,(e^x-1)=xe^x$ for $G(x)=\sum_{j\ge0}C_jx^j/j!$, so
$G(x)=xe^x/(e^x-1)$, the generating function of $B_j(1)$
\cite[Eq.~(24.2.3)]{DLMF}. From $G(-x)=x/(e^x-1)=G(x)-x$ one reads off
$(-1)^jC_j=C_j$ for $j\ge2$ and $-C_1=C_1-1$, which is
Theorem~\ref{thm:odd} again. The proof via Lemma~\ref{lem:refl} avoids
generating functions altogether.
\end{remark}

\begin{remark}[General progressions]
For general $a_0$ and $d$ the odd constants do not vanish; for example
$a_0=2$, $d=1$ gives $C_3=3$. For $d\ne0$, by
Theorem~\ref{thm:ident} and the reflection formula
$B_m(1-x)=(-1)^mB_m(x)$ \cite[Eq.~(24.4.3)]{DLMF},
$C_m(d-a_0,d)=(-1)^mC_m(a_0,d)$; taking $x=\tfrac12$ there gives
$B_m(\tfrac12)=0$ for odd $m$, so the constants $C_3,C_5,\dots$ vanish when
$a_0/d\in\{0,\tfrac12,1\}$, for instance for the sum
$0^m+1^m+\dots+(n-1)^m$.
\end{remark}

\begin{remark}[Computation]
Theorem~\ref{thm:odd} is a proved property of the constants
$C_m=S_m'(0)$, which by Corollary~\ref{cor:rbef} are the unique solution of
the recurrence; it is not an assumption. An implementation may skip the odd constants in the normalized
case, but a verification program should compute them and confirm that they
vanish.
\end{remark}

\section{Sequences with Polynomial Differences}
\label{sec:polydiff}

Let $(x_n)_{n\ge1}$ satisfy $x_1=a_0$ and $x_{n+1}-x_n=D(n)$, where $D\not\equiv0$ is a
polynomial of degree $\ell$. By Theorem~\ref{thm:ext} applied to $D$,
$x_n=a_0+\sum_{i=1}^{n-1}D(i)$ is a polynomial in $n$; its degree is
$q=\ell+1$, because the leading term of $D$ contributes the leading term of
the sum, and
Theorem~\ref{thm:int} applies to $S(n)=\sum_{i=1}^{n}x_i$ directly.

\begin{proposition}[Procedure]\label{prop:proc}
Write $x'(t)=\sum_{k=0}^{q-1}e_k\,t^k$ and let $\sigma_k(n)=\sum_{i=1}^{n}i^k$
be the normalized power sums of Section~\ref{sec:rbef} ($\sigma_0(n)=n$). Then
\[
  U=\Sigma[x']=\sum_{k=0}^{q-1}e_k\,\sigma_k,\qquad
  S(n)=\int_0^{n}U(t)\,dt+c\,n,\qquad c=a_0-\int_0^{1}U(t)\,dt .
\]
\end{proposition}
\begin{proof}
The first identity is linearity of $\Sigma[\cdot]$; the rest is
Theorem~\ref{thm:int} with $x(1)=a_0$.
\end{proof}

\begin{remark}
Summing $x_n$ term by term uses the power sums $\sigma_0,\dots,\sigma_q$.
The procedure takes only $\sigma_0,\dots,\sigma_{q-1}$ as input; the
integration step does the work of the highest power sum (it is
Corollary~\ref{cor:power} again), and its linear coefficient comes from the
single condition $S(1)=a_0$. This is the same mechanism by which the constants
$C_m$ arise in Section~\ref{sec:rbef}.
\end{remark}

\subsection{Example: cubic differences}
To avoid clashes with the common difference $d$ of earlier sections, write
\[
  D(i)=\alpha+\beta i+\gamma i^2+\delta i^3 .
\]
Then, with the classical values of $\sigma_1,\sigma_2,\sigma_3$,
\begin{equation}\label{eq:xcubic}
  x_n=a_0+\alpha(n-1)+\frac{\beta}{2}(n-1)n+\frac{\gamma}{6}(n-1)n(2n-1)
      +\frac{\delta}{4}(n-1)^2n^2 ,
\end{equation}
\begin{equation}\label{eq:xprime}
  x'(n)=\delta n^3+\Bigl(\gamma-\frac32\delta\Bigr)n^2
        +\Bigl(\beta-\gamma+\frac{\delta}{2}\Bigr)n
        +\Bigl(\alpha-\frac{\beta}{2}+\frac{\gamma}{6}\Bigr).
\end{equation}
Proposition~\ref{prop:proc} gives
$U=\delta\sigma_3+(\gamma-\tfrac32\delta)\sigma_2+(\beta-\gamma+\tfrac{\delta}{2})\sigma_1
+(\alpha-\tfrac{\beta}{2}+\tfrac{\gamma}{6})\sigma_0$, and integration yields
\begin{equation}\label{eq:Scubic}
  S(n)=\frac{\delta}{20}n^5+\frac{\gamma}{12}n^4
      +\Bigl(\frac{\beta}{6}-\frac{\delta}{12}\Bigr)n^3
      +\Bigl(\frac{\alpha}{2}-\frac{\gamma}{12}\Bigr)n^2
      +c\,n,
  \qquad
  c=a_0-\frac{\alpha}{2}-\frac{\beta}{6}+\frac{\delta}{30}.
\end{equation}
Formula \eqref{eq:Scubic} agrees identically with
$\sum_{i=1}^{n}x_i$ computed from \eqref{eq:xcubic}.

\section{Computational Verification}
\label{sec:verif}

The recurrence \eqref{eq:rec} and the closed form \eqref{eq:closedA} were
implemented in exact rational arithmetic (Appendix~\ref{app:code}). The
program computes every constant, including those with odd index, and
performs the following checks, all of which passed for $M=300$.
\begin{enumerate}
\item For $3\le m\le300$, the coefficient row generated by symmetry from
the form of Proposition~\ref{prop:pascal} equals both the definition
$\binom mk/(k+1)$ and the row produced by the ratio rule of
Section~\ref{sec:row}.
\item For $a_0=d=1$, $C_m$ equals the Bernoulli number $B_m$ computed
independently with SymPy 1.14, for $2\le m\le300$. ($C_1=\tfrac12$ holds by
construction; SymPy's \texttt{bernoulli(1)} returns $+\tfrac12$, which does
not affect checks that start at $m=2$.)
\item The $149$ odd-index constants $C_3,C_5,\dots,C_{299}$ are computed,
not imposed, and all vanish (Theorem~\ref{thm:odd}).
\item For every even $m$ with $2\le m\le300$, $C_m$ satisfies the von
Staudt--Clausen theorem $C_{m}+\sum_{p\ \mathrm{prime},\,(p-1)\mid m}1/p
\in\mathbb Z$ \cite[Eq.~(24.10.1)]{DLMF}.
\item For $25$ random rational pairs $(a_0,d)$, the closed form
\eqref{eq:closedA} reproduces the direct sum for $1\le m\le40$ and all
$1\le n\le m+2$; since both sides are polynomials in $n$ of degree at most
$m+1$, this proves the identity for these $(a_0,d)$ and $m$.
\item For $6$ further random pairs and $2\le m\le25$,
$C_m=d^mB_m(a_0/d)$ (Theorem~\ref{thm:ident}).
\end{enumerate}

\section{Conclusion}
\label{sec:concl}

For sums whose terms are polynomial in the index, the polynomial extension
turns the difference equation of the sum into an identity that may be
differentiated and integrated exactly. Applied to power sums of an
arithmetic progression, differentiation lowers the power by one and
integration restores it with a single constant, $C_m=S_m'(0)$, fixed by
$S_m(1)=a_0^m$. This yields the closed form and the recurrence of
Section~\ref{sec:rbef} without induction on the exponent and without
generating functions. With the Bernoulli numbers defined by their classical recurrence, which the
progression starting at $0$ reproduces exactly, the constants are
$d^mB_m(a_0/d)$, and $B_m(1)$ in
the normalized case; the value $C_1=+\tfrac12$ reflects the summation
range, and the vanishing of the
odd constants follows from a reflection identity. The same mechanism gives
closed forms for sequences with polynomial differences. The formulas are
classical in content, and in the normalized case the integral recursion is
known \cite{Cereceda,Coen}; the contribution of the paper is a rigorous
foundation for it, its extension to arbitrary arithmetic progressions, and
derivations in which each constant is explained.

The method requires the terms to be polynomial in the index. For other
sequences the partial sums need not be polynomial: for reciprocals there is
no polynomial extension, and the uniqueness on which Section~\ref{sec:why}
rests is lost. For geometric terms the partial sums have a natural
extension of exponential type, and an analogue of the method for that class
may be possible. A natural extension is to alternating sums
$\sum_{i=1}^{n}(-1)^{i-1}(a_0+(i-1)d)^m$, whose classical closed form involves
Euler polynomials \cite[Eq.~(24.4.10)]{DLMF}; the same approach, with the
forward difference replaced by the operator $p(t+1)+p(t)$, is expected to
produce the corresponding constants as integration constants.

\appendix
\section{Reference Implementation}
\label{app:code}

The following Python program implements Proposition~\ref{prop:pascal}, the
recurrence \eqref{eq:rec} and the closed form \eqref{eq:closedA}, and runs
the checks of Section~\ref{sec:verif}. It requires Python~3.8 or later and, for the
checks, SymPy.

{\footnotesize
\begin{verbatim}
"""
RBEF reference implementation and verification.

generator_row(m)       : coefficient row A(m,k) = binom(m,k)/(k+1), k=1..m-2,
                         built from one half by A(m,k) = A(m,m-1-k).
rbef_constants(M,a0,d) : constants C_1..C_M by the recurrence of
                         Corollary 4.3. Every constant is computed; the
                         odd-index ones are not skipped, so their
                         vanishing in the normalized case is checked.
verify(M)              : compares with independent references and
                         rebuilds S_m(n) for general a0, d.

Usage: python3 rbef_reference.py [M]   (default M = 300)
Requires Python >= 3.8 (math.comb); the checks need sympy.
"""
import math
import random
import sys
from fractions import Fraction as F


# ------------------------------------------------------ coefficient row
def generator_row(m):
    """[A(m,1), ..., A(m,m-2)], A(m,k) = binom(m+1,k+1)/(m+1)."""
    n = m - 2
    if n < 1:
        return []
    row = [None] * n                  # row[k-1] = A(m,k)
    for k in range(1, n // 2 + 1):    # one half ...
        v = F(math.comb(m + 1, k + 1), m + 1)
        row[k - 1] = v
        row[(m - 1 - k) - 1] = v      # ... and its mirror image
    if n % 2 == 1:                    # a centre exists when m-2 is odd
        k = (n + 1) // 2
        row[k - 1] = F(math.comb(m + 1, k + 1), m + 1)
    assert all(x is not None for x in row)
    return row


# ------------------------------------------------------ RBEF constants
def rbef_constants(M, a0=F(1), d=F(1)):
    a0, d = F(a0), F(d)
    C = {1: a0 - d / 2}
    for m in range(2, M + 1):
        row = generator_row(m)
        s = F(0)
        for k in range(1, m - 1):
            s += row[k - 1] * d ** (m - k - 1) * C[k + 1]
        C[m] = a0 ** m - d ** m / (m + 1) - d ** (m - 1) * C[1] - s
    return C


def power_sum_from_rbef(m, n, a0, d, C):
    a0, d, n = F(a0), F(d), F(n)
    tot = d ** m * n ** (m + 1) / (m + 1) + d ** (m - 1) * C[1] * n ** m
    row = generator_row(m)
    for k in range(1, m - 1):
        tot += row[k - 1] * d ** (m - k - 1) * C[k + 1] * n ** (m - k)
    if m >= 2:
        tot += C[m] * n               # the k = m-1 term, A(m,m-1) = 1
    return tot


def direct_power_sum(m, n, a0, d):
    return sum((F(a0) + F(d) * (i - 1)) ** m for i in range(1, n + 1))


# ------------------------------------------------------ verification
def _primes_upto(n):
    sieve = bytearray([1]) * (n + 1)
    sieve[:2] = b"\x00\x00"
    for i in range(2, int(n ** 0.5) + 1):
        if sieve[i]:
            sieve[i * i::i] = bytearray(len(sieve[i * i::i]))
    return [i for i, v in enumerate(sieve) if v]


def _frac(r):
    """sympy Rational -> Fraction"""
    return F(int(r.p), int(r.q))


def verify(M=300, seed=1):
    import sympy as sp
    ok_all = True

    def report(name, ok, extra=""):
        nonlocal ok_all
        ok_all &= bool(ok)
        print(f"[{'PASS' if ok else 'FAIL'}] {name} {extra}")

    # 1. row versus the definition binom(m,k)/(k+1) and the ratio rule
    bad = []
    for m in range(3, M + 1):
        defn = [F(math.comb(m, k), k + 1) for k in range(1, m - 1)]
        ratio = [F(m, 2)]
        for k in range(1, m - 2):
            ratio.append(ratio[-1] * F(m - k, k + 2))
        if not (generator_row(m) == defn == ratio):
            bad.append(m)
    report("row = binom(m,k)/(k+1) = ratio rule", not bad,
           f"(m=3..{M}; bad: {bad[:5]})")

    # 2. normalized constants versus Bernoulli numbers
    C = rbef_constants(M)
    wrong = [m for m in range(2, M + 1)
             if C[m] != _frac(sp.bernoulli(m))]
    report("C_m = B_m, 2<=m<=M", not wrong, f"(M={M}; bad: {wrong[:5]})")
    report("C_1 = a0 - d/2 = +1/2 (by construction)", C[1] == F(1, 2))

    # 3. odd constants vanish although nothing forces them to
    odd = [m for m in range(3, M + 1, 2) if C[m] != 0]
    report("odd C_m = 0 (computed)", not odd,
           f"({len(range(3, M + 1, 2))} values)")

    # 4. von Staudt-Clausen
    vsc_bad = []
    for m in range(2, M + 1, 2):
        ps = [p for p in _primes_upto(m + 1) if m % (p - 1) == 0]
        if (C[m] + sum(F(1, p) for p in ps)).denominator != 1:
            vsc_bad.append(m)
    report("von Staudt-Clausen", not vsc_bad)

    # 5. closed form versus direct sums for general a0, d
    rng = random.Random(seed)
    Mp = min(M, 40)
    bad_ps = 0
    for _ in range(25):
        a0 = F(rng.randint(-9, 9), rng.randint(1, 5))
        d = F(rng.choice([x for x in range(-7, 8) if x]),
              rng.randint(1, 4))
        Cg = rbef_constants(Mp, a0, d)
        for m in range(1, Mp + 1):
            for n in range(1, m + 3):     # m+2 points fix the polynomial
                if (power_sum_from_rbef(m, n, a0, d, Cg)
                        != direct_power_sum(m, n, a0, d)):
                    bad_ps += 1
    report("closed form = direct sum", bad_ps == 0,
           f"(25 pairs, m=1..{Mp}, n=1..m+2; failures: {bad_ps})")

    # 6. C_m(a0,d) = d^m B_m(a0/d)
    x = sp.Symbol("x")
    rel_bad = 0
    for _ in range(6):
        a0 = F(rng.randint(1, 9), rng.randint(1, 4))
        d = F(rng.randint(1, 7), rng.randint(1, 3))
        Cg = rbef_constants(25, a0, d)
        y = sp.Rational(a0.numerator, a0.denominator) / \
            sp.Rational(d.numerator, d.denominator)
        dd = sp.Rational(d.numerator, d.denominator)
        for m in range(2, 26):
            if Cg[m] != _frac(dd ** m * sp.bernoulli(m, x).subs(x, y)):
                rel_bad += 1
    report("C_m = d^m B_m(a0/d)", rel_bad == 0, f"(failures: {rel_bad})")

    print("\nOVERALL:", "all checks passed" if ok_all else "FAILURES")
    return ok_all


if __name__ == "__main__":
    verify(int(sys.argv[1]) if len(sys.argv) > 1 else 300)
\end{verbatim}
}

\end{document}